\documentclass[12pt]{article}

 \usepackage[ansinew]{inputenc}
 \usepackage[all]{xy}
\newdir{ >}{!/8pt/@{}*@{>}}

\usepackage{amssymb, amsmath,amsthm}

\theoremstyle{plain}
\newtheorem{Teorema}{Theorem}[section]
\newtheorem{Lemma}[Teorema]{Lemma}
\newtheorem{Corollario}[Teorema]{Corollary}
\newtheorem{Proposizione}[Teorema]{Proposition}

\theoremstyle{definition}
\newtheorem{Definizione}[Teorema]{Definition}

\theoremstyle{remark}

\newtheorem{Osservazione}[Teorema]{Remark}

\newtheorem{Esempio}[Teorema]{\bf Example}

\newcommand{\topo}{\operatorname{\mathbf{Top}}}
\newcommand{\topgen}{\mathcal T}
\newcommand{\gsecat}{\operatorname{{\rm Gsecat}}}
\newcommand{\wsecat}{\operatorname{{\rm Wsecat}}}
\newcommand{\cat}{\operatorname{{\mathbf{cat}}}}
\newcommand{\tc}{\operatorname{{\mathbf{TC}}}}
\newcommand{\ctc}{\operatorname{{\rm TC}}}
\newcommand{\ccat}{\operatorname{{\rm cat}}}
\newcommand{\csecat}{\operatorname{{\rm secat}}}
\newcommand{\secat}{\operatorname{{\mathbf{secat}}}}

\newcommand{\simpset}{\operatorname{{\mathcal S\mathcal S}{et}}}
\newcommand{\singular}{\operatorname{\rm Sing}}
\newcommand{\Id}{\operatorname{\rm Id}}

\title{Abstract sectional category in model structures on topological spaces}
\author{Marco Moraschini\footnote{Partially supported by the Extra project  of the University of Milano-Bicocca.}\ \ and Aniceto Murillo\footnote{Partially supported by the MINECO grant MTM2013-41768-P and  the Junta de Andaluc{\'\i}a grant FQM-213.}}

\begin{document}

\maketitle

\begin{abstract}
We study the behavior of the abstract sectional category in the Quillen, the Str{\o}m and the Mixed proper model structures on  topological spaces and prove that, under certain reasonable conditions, all of them coincide with the classical notion. As a result, the same conclusions hold for the abstract Lusternik-Schnirelmann category and the abstract topological complexity of a space.
\end{abstract}

\section*{Introduction}

It is well known that most numerical homotopy invariants of Lusternik-Schnirelmann type on topological spaces are derived from the {\em sectional category} or {\em genus} of a map, introduced by Schwarz in    \cite{sch}. For a given map $f\colon X\to Y$, its sectional category $\secat (f)$ is the smallest integer $n$  such that $Y$ can be covered by $n+1$ open subsets on each of which $f$ admits a homotopy local section.
As this invariant appears in different settings, and in order to study it under a common viewpoint, in  \cite{digargarmure} the authors introduced the abstract sectional category  in a $J$-category, for instance a (proper) closed model category satisfying the so called cube lemma (see next sections for precise definitions). This is done via two different notions, the {\em Ganea sectional category} $\gsecat(f)$ and the {\em Whitehead sectional category} $\wsecat (f)$ of a morphism $f$  on a general model category $\mathcal C$, which are shown to coincide, in a $J$-category \cite[Thm. 15]{digargarmure}. Both equivalent invariants are denoted simply by $\csecat(f)$, or $\csecat^{\mathcal C}(f)$ if we want to stress in which category we are working.   Then, in the topological setting, the classical sectional category of a map is precisely its abstract sectional category  on the model structure of topological spaces in the sense of Str{\o}m \cite{s} (see next sections for precise definitions).

 Our main goal is to show that, under sufficiently general conditions,  one can also recover this invariant from the abstract sectional category on the original Quillen model category structure \cite{qui}. Indeed, from now on, denote indistinctly by  $\topgen$ either the category $\topo$ or $\topo_*$ of free and pointed topological spaces. Also, denote  by $\topgen_S$, and $\topgen_Q$  the Str{\o}m and Quillen model structure respectively in $\topgen$. Then, we prove (see Theorem \ref{principal} for a more precise statement):

\begin{Teorema}\label{principalintro}  Let $f\colon X\to Y$ be a map in which $X,Y$ are of the homotopy type of locally compact CW-complexes. Then,
$$
\csecat^{\topgen_S}(f)=\csecat^{\topgen_Q}(f).
$$
In particular, in the free setting and with $Y$ normal, any of these invariant yields  $\secat(f)$.
\end{Teorema}

As a consequence, we show that these common invariants have a simplicial, and therefore, combinatorial description. Indeed, Let $\simpset$ denote  the category of either free or pointed simplicial sets. Using the classical Milnor equivalence we deduce the following (see the first assertion of Theorem \ref{ultimo} for a precise and more general statement).

\begin{Teorema}\label{uno} Let $f\colon X\to Y$ be a map in which $Y$ is normal and $X, Y$ have the homotopy type of locally compact CW-complexes. Then,
$$
\secat(f)=\csecat^{\simpset}\bigl(\singular(f)\bigr).
$$
\end{Teorema}

These results readily translate to widely studied LS invariants. Recall that the {\em Lusternik-Schnirelmann category} $\cat(X)$ of a given topological space \cite{lussch} is the least integer $n$ for which $X$ can be covered by $n+1$ open sets deformable to a point within $X$.  On the other hand,  the \textit{topological complexity} $\tc(X)$ of  $X$  \cite{far} is the least integer $n$ such that $X\times X$ can be covered by $n+1$ open sets on each of which there is a section of the path fibration $X^I\to X\times X$ which associates to each path its initial and final points. When $X$ is  the configuration space associated to the motion of a given mechanical system, this invariant measures, roughly speaking, the minimum amount of instructions needed for controlling the given system.

In a  $J$-category $\mathcal C$, the {\em abstract Lusternik-Schnirelmann category} $\ccat^{\mathcal C}(X)$ of a given object $X$  is defined as $\csecat^{\mathcal C}(*\to X)$ where $*$ denotes the $0$-object of $\mathcal C$ \cite{do}. In the same way, the {\em abstract topological complexity} $\ctc^{\mathcal C}(X)$ is defined as $\csecat^{\mathcal C}(\Delta)$ where $\Delta\colon X\to X\times X$ is the diagonal \cite{digargarmure}. Then,  the path fibration is homotopy equivalent to the diagonal, and for path-connected spaces one has
  $$
  \cat(X)=\ccat^{\topgen_S}(X),\qquad  \tc(X)=\ctc^{\topgen_S}(X).
   $$
   For the first equality $X$ only needs to be normal and well pointed \cite[Thm 1.55]{corlupotan}, and $\topgen=\topo_*$. For the second, $\topgen=\topo$ and we  require $X\times X$ to be normal \cite[Thm. 2.2]{ferghienkahlvan}. Under these general hypotheses, the theorems above  immediately imply (this is Corollary \ref{elunico} and the second assertion of Theorem \ref{ultimo}):

\begin{Corollario}\label{dos} For any space $X$  of the homotopy type of a locally compact CW-complex,
$$
\begin{aligned}
\tc(X)&=\ctc^{\topgen_Q}(X)=\ctc^{\simpset}\bigl(\singular(X)\bigr),\\
\cat(X)&=\ccat^{\topgen_Q}(X)=\ccat^{\simpset}\bigl(\singular(X)\bigr).
\end{aligned}
$$
\end{Corollario}

 The next section is devoted to set the notation and present basic facts. In Section 2 we prove the results above.

We thank the referee for his/her valuable remarks, corrections and suggestions which have improved both the content and presentation of the paper.

\section{Preliminaries}

In this section  we present the basic constructions and some results about the abstract sectional category of a given morphism in a categorical setting \cite{digargarmure}.

We follow the usual modern approach, see for instance \cite{ho}, and drop the original adjective ``closed" introduced by Quillen in \cite{qui} when considering a model category $\mathcal C$. For such a category we denote by $fib$, $cof$, and $we$ the family of fibrations, cofibrations and weak equivalences respectively. A factorization of a given morphism $f=p\circ h$ as a composition of a weak equivalence and a fibration is called an {\em $F$-factorization} of $f$. In the same way, a {\em $C$-factorization} of $f$ is a composition $f=h\circ i$ of a cofibration and a weak equivalence.  Recall that a model category is {\em proper} \cite{baues} if weak equivalences are preserved by pullbacks along fibrations and by pushouts along cofibrations. From now on, every considered category shall be a proper model category.  Remark also that such a category, if pointed, satisfies all the axioms of a so called $J$-category (see \cite{digargarmure} or \cite{do} for a precise definition) except that it may not satisfy the {\em cube axiom} \cite[\S1]{do}.

Given a category $\mathcal C$ endowed with two proper model structures $(fib_{1}, cof_{1},$ $ we_{1})$ and $(fib_{2}, cof_{2}, we_{2})$ such that ${we}_{1} \subset we_{2}$ and $fib_{1} \subset fib_{2}$, there  exists another proper model structure $(fib_{m}, cof_{m}, we_{m})$ on $\mathcal C$,  called \textit{mixed structure}, for which $we_{m} = we_{2}$ and $fib_{m} = fib_{1}$ \cite[Thm. 2.1 and \S4]{co} (note that in this reference, the classical Quillen ``closed" terminology is used).

As there are slightly different approaches to homotopy pullbacks and pushouts in model categories, we include here the one we use. A commutative square
$$\xymatrix{
  {D} \ar[d]^{g'} \ar[r]^{f'}
                      & C \ar[d]_{g}    \\
  {A} \ar[r]_{f}     & B               }
$$ \noindent is  a \textit{homotopy pullback} if given
an
$F$-factorization of $g$ (equivalently $f$ or both),  the dotted induced map %$\tau '$
$D\stackrel{\sim}{\to}E'$ is a weak equivalence
$$\xymatrix@C=.7cm@R=.5cm{
{D} \ar[dd]_{g'} \ar[rrr]^{f'} \ar@{.>}[dr]_{\sim} & & & {C} \ar[dd]^g
\ar[dl]^{\tau }_{\sim} \\
 & {E'} \ar@{>>}[dl]_{\overline{p}} \ar[r]_{\overline{f}} & {E} \ar@{>>}[dr]^p & \\
 {A} \ar[rrr]_f & & & B.  }$$
 Here, $E'$ denotes the pullback of $p$ and $f$. This definition does not depend on the chosen factorizations.
The notion of \textit{homotopy pushout} is dually defined.

In a proper model category the \textit{join} $A \ast_{B} C$ of two morphisms $f\colon A \to B$ and $g\colon C \rightarrow B$ is the object obtained by the pushout of $i$ and $\pi_2$ in the following diagram,
\begin{displaymath}
\xymatrix@!0{
 & Z \times_{B} C \ar[rrrr]^{\pi_1}\ar@{>>}[ddd]_{\pi_2}\ar@{ >->}_{i}[rrd]
   && && Z \ar@{>>}^{p}[ddd]\ar@{<-}_{\sim}^{\tau}[lld] & A \ar[l]^{\nu}_{\sim} \ar[lddd]^{f}
\\
 & && M \ar[d]_{\pi_2'} && &
 \\
 & && A\ast_{B} C \ar@{.>}^{j}[rrd] && &
 \\
  & C \ar@{ >->}[rru]^{i'} \ar[rrrr]_{g} && && B, &
 }
\end{displaymath}
where $f = p \circ \nu$ is any $F$-factorization, $Z\times_BC$ is the pullback of $g$ and $p$ and $\pi_1 = \tau \circ i$ is any $C$-factorization.
The dotted map induced by the pushout construction from $A\ast_{B} C$ to B is called the \textit{join morphism} of $f$ and $g$. The join $A\ast_{B} C$ is well defined and symmetric up to weak equivalence. Also, the weak equivalence class of the join map does not depend on the chosen factorization.

On the other hand \cite[Def. 4]{digargarmure} a morphim $g\colon C\to B$ is said to admit a  \textit{weak section} if, given an $F$-factorization $g=p\circ \tau$ of $g$ and a cofibrant replacement $\overline B\stackrel{\sim}{\twoheadrightarrow} B$ of $B$ (i.e., a trivial fibration  in which $\overline B$ is cofibrant), there exists $s$ making commutative the following diagram:
\begin{equation}\label{diagrama2}
 \xymatrix@!0{
 && C \ar[ld]_{\sim}^{\tau} \ar[dd]^{g} \\
 & E \ar@{>>}[rd]^{p} \ar@{<.}[ld]_{s}\\
 \overline{B} \ar@{>>}[rr]_{{\sim}} && B.
 }
\end{equation}
This definition does not depend on the made choices.

\begin{Definizione}\cite[Def. 8]{digargarmure}
Let $f \colon X \rightarrow Y$ be any morphism in $\mathcal C$. For each $n\ge 0$ we define the object $\ast_{Y}^{n} X$ and the morphism $h_{n} \colon \ast_{Y}^{n} X \rightarrow Y$ inductively as follows:

\textit{(i)} $h_{0} = f \colon X \rightarrow Y$ (and so $\ast_{Y}^{0} X = X$).

\textit{(ii)} Suppose that $h_{n-1} {\colon} \ast_{Y}^{n-1}X \rightarrow Y$ has already been  constructed. Then, $h_{n}$ is the join morphism of $f$ and $h_{n-1}$.

The \textit{Ganea sectional category} of $f$, denoted by $\gsecat(f)$, is the least integer $n \leq \infty$ such that $h_{n}$ admits a weak section.
\end{Definizione}
This notion is an invariant of weakly equivalent morphisms \cite[Proposition 10]{digargarmure}. In other words, if $f$ and $g$ are connected by a chain of weak equivalences in the morphism category, then $\gsecat(f)=\gsecat(g)$.

\begin{Osservazione}
The {\em Whitehead sectional categoy} $\wsecat(f)$ of a given morphism can also be defined by means of the abstract fat wedge \cite[Def. 12]{digargarmure}. Moreover, if the proper model structure in $\mathcal C$ satisfies the cube axiom, then \cite[Thm. 15]{digargarmure},
$$
\gsecat(f)=\wsecat(f)
$$ for any morphism $f$, and we denote this common invariant simply by
$$\csecat(f)\quad\text{or}\quad
\csecat^{\mathcal C}(f)$$ if we want to stress the category in which we are working. Nevertheless, we have chosen the Ganea approach as it fits better in our arguments.
\end{Osservazione}

Next, we recall that a covariant functor $\mu {\colon} {\mathcal C} \rightarrow {\mathcal D}$  is said to be a \textit{modelization functor} \cite[Def. 6.1]{do} if it preserves weak equivalences, homotopy pullbacks and homotopy pushouts.
The following helps to detect modelization functors and when such functors preserve the abstract sectional category in any of its versions.

\begin{Proposizione}\label{dim-mod-funt}{\em \cite[Prop. 6.5]{do}}
Let $\alpha {\colon} {\mathcal C} \leftrightarrows {\mathcal D} {\colon} \beta$ be a pair of adjoint functors ($\alpha$ is left adjoint) such that:

(i) $\alpha$ and $\beta$ preserve weak equivalences.

(ii) $\alpha$ preserves cofibrations and $\beta$ preserves fibrations.

 (iii) The adjunction maps $X \rightarrow \beta \circ \alpha (X) $ and $\alpha \circ \beta (Y)  \rightarrow Y$ are weak equivalences for any objects $X, Y$.

Then, $\alpha$ and $\beta$ are modelization functors.
\end{Proposizione}

\begin{Teorema}\label{teor-TC-uguale}{\em \cite[Rem. 25 and Cor. 26]{digargarmure}}
Let $\mu {\colon} {\mathcal C} \leftrightarrows {\mathcal D} {\colon} \nu$ be modelization functors  and let $f$ be a morphism in ${\mathcal C}$ such that $\nu\bigl(\mu(f)\bigr)$ is weakly equivalent to $f$. Then, \begin{displaymath}\gsecat^{\mathcal D}\bigl(\mu(f)\bigr)=\gsecat^{\mathcal C}(f)\quad\text{and}\quad \wsecat^{\mathcal D}\bigl(\mu(f)\bigr)=\wsecat^{\mathcal C}(f).\end{displaymath}
\end{Teorema}

Finally,  recall that  the {\em abstract Lusternik-Schnirelmann category} $\ccat(X)$, or $\ccat^{\mathcal C}(X)$, of a given object $X$ in $\mathcal C$ is defined as \cite[\S3]{do},
$$
\ccat(X)=\csecat(*\to X),
$$
where $*$ denotes the initial object.
In the same way, the {\em abstract topological complexity} $\ctc(X)$, or $\ctc^{\mathcal C}(X)$, of a given object $X$ is defined as \cite[\S3.1]{digargarmure},
$$
\ctc(X)=\csecat(\Delta),
$$
where $\Delta\colon X\to X\times X$ is the diagonal.

\begin{Osservazione} Here, as $\mathcal C$ is not assumed to be a $J$-category, i.e., it might not be pointed, we have to specify the initial object in the definition of $\ccat(X)$. In the next section, see Remark \ref{veremos} below, when we identify the classical Lusternik-Schnirelmann category of a topological space with its abstract counterpart, we do so  only in the based setting $\topo_*$ where the point is the initial object. In the same way, we will identify the usual topological complexity of a topological space only in the free setting $\topo$. In this case, whenever Theorem \ref{teor-TC-uguale} is invoked, we will be using its ``unpointed'' version \cite[Thm. 27 and Cor. 28]{digargarmure} for which the modelization functors are required to preserve the final object.
\end{Osservazione}

\section{Sectional category in model structures on topological spaces }

Recall that in $\topo$, the category of topological spaces and continuous maps, there are two standard structures which make it a proper model category. The first one, in which the fibrations are the Serre fibrations, the weak equivalences are the weak homotopy equivalences and the cofibrations are those morphisms that have the left lifting property with respect to all trivial fibrations, was defined by Quillen in \cite{qui}. The other, introduced by  Str\o{}m in \cite{s}, has as fibrations the Hurewicz fibrations, as cofibrations the closed topological cofibrations and as weak equivalences the homotopy equivalences. These structures are inherited by $\topo_*$, the category of pointed spaces and pointed maps. Unless explicitely stated otherwise, we will make no distinction  between the free and pointed setting and denote either $\topo$ or $\topo_*$ by $\mathcal T$. In the same way $\topgen_S$ and $\topgen_Q$ will denote the corresponding Str{\o}m or Quillen proper structure.

Associated to these structures we can also consider the proper model caegory given by the {\em mixed structure} both in the free and pointed setting, which we denote by $\topgen_M$. For it, the mixed fibrations are the Hurewicz fibrations, the mixed weak equivalences are the weak homotopy equivalences and the mixed cofibrations are those morphisms which have the left lifting property with respect to all trivial fibrations. It has been shown \cite[Ex. 3.8]{co}, \cite[Cor. 17.4.3]{pomay}, that the cofibrant objects in $\topgen_{M}$ are exactly the (well pointed in the based setting) topological spaces with the homotopy type of a CW-complex.

\begin{Osservazione}\label{porfin} In the following, and to assure the cofibrant character  on the mixed structure, any space of the homotopy type of a CW-complex shall be well pointed whenever we refer to the based setting.
\end{Osservazione}

 The following result is crucial for our purposes.

\begin{Lemma}\label{join-inv2}
Let $f \colon A \rightarrow B$ and $g \colon C \rightarrow B$ be two maps in which $A, \, B$ and $C$ have the homotopy type of  CW-complexes. Then, the join object $A \ast_{B} C$ computed in $\mathcal{T}_{S}$ can be chosen to be the same as the one obtained in $\mathcal{T}_{M}$.
\end{Lemma}
%We have to say in the section preliminaries that the join object is unique up to weak equivalence and it is simmetric.
%We have to change the join diagram in the preliminaries since the have to exchange the role of f and g.
\begin{proof}
To build the join in the Str\o{}m  structure, choose an $F$-factorization $f = p \circ \nu$ of $f$, where $\nu$ is a homotopy equivalence and $p$ is a Hurewicz fibration. Then, observe that this is also an $F$-factorization in the mixed structure, since the class of the fibrations is the same in both structures and every homotopy equivalence is also a weak homotopy equivalence. In particular, the factorizing object $Z$ as in diagram (\ref{diagrama1}) below turns out to be, by the observation above, a cofibrant object in the mixed structure.

The second step is to consider the pullback $Z \times_{B} C$ of $p$ and $g$, which is the same object in both  model structures. Note that $Z \times_{B} C$ has the homotopy type of a CW-complex (see for instance \cite[Thm. 7.5.9]{se}).

Next, we  take a $C$-factorization of the projection $\pi_{1} {\colon} Z \times_{B} C \rightarrow Z$ in the Str\o{}m structure  through its mapping cylinder $M_{\pi_{1}}$. That is, $\pi_{1} = \tau \circ i$ in which $i  \colon Z \times_{B} C \hookrightarrow M_{\pi_{1}}$ is a closed topological cofibration and $\tau \colon M_{\pi_{1}} \stackrel{\sim}{\rightarrow} Z$ is a homotopy equivalence.

The key point is that this is in fact a $C$-factorization of $\pi_1$ in  the mixed structure. For it observe that $M_{\pi_{1}}$ has the same homotopy type of a CW-complex as $\tau$ is a  equivalence. Hence, it follows that $i$ is a cofibration in the Str\o{}m structure between objects that are cofibrant in the mixed structure. We now apply \cite[Cor. 3.12]{co} by which every cofibration in the Str\o{}m structure between mixed cofibrant objects is also a mixed cofibration. Thus, $i$ is  a mixed cofibration and therefore, $\pi_{1} = \tau \circ i$ is also a $C$-factorization in the mixed structure. Hence, the following construction
\begin{equation}\label{diagrama1}
\xymatrix@!0{
Z \times_{B} C \ar@{ >->}[drr]_{i} \ar@{>>}[ddd]_{\pi_{2}}  \ar[rrrr]^-{\pi_{1}} &&&& Z \ar@{>>}[ddd]^{p} & A \ar[l]_{\nu}^{\sim} \ar[lddd]^{f} \\
&& M_{\pi_{1}} \ar[urr]_{\tau}^{\sim} \ar[d]_{\pi_{2} '}\\
&& A \ast_{B} C \ar[rrd]^{j} \\
C \ar@{ >->}[rru]^{i'} \ar[rrrr]_{g} &&&& B
}
\end{equation}
holds in both  model structures and the claim follows.
\end{proof}

\begin{Osservazione}\label{oss-cl-push}
Observe that since \cite[\S II, Prop. 8.1]{gojar} holds, the join object $A \ast_{B} C$ constructed in the previous proof is a cofibrant object in the mixed structure, that is \cite{co},\cite[Cor. 17.4.3]{pomay}, it has the homotopy type of a CW-complex.
Moreover, the join morphism $j \colon A \ast_{B} C \rightarrow B$ constructed in diagram (\ref{diagrama1}) above is the same in the two model structures, since it only depends on the universal property of pushouts.
\end{Osservazione}

\begin{Lemma}\label{weak-sec-uguale-str-mix1}
Let $A,\,B,\,C$ be spaces of the homotopy type of CW-complexes and let $j \colon A \ast_{B} C \rightarrow B$ be the join morphism constructed above. Then, $j$ admits a weak section in $\mathcal{T}_{S}$ if and only if it admits a weak section in $\mathcal{T}_{M}$.
\end{Lemma}
\begin{proof}
It is clear that if $j$ admits a weak section in the Str\o{}m  structure, it also admits a weak section in the mixed one, since every $F$-factorization in the Str\o{}m structure is also an $F$-factorization in the mixed one.

Conversely, suppose that $j$ admits a weak section in the mixed structure. Recall that $A \ast_{B} C$ is cofibrant in the mixed structure by Remark \ref{oss-cl-push} and observe that  it is always possible to choose an $F$-factorization through a cofibrant object.  Then, by the Whitehead Theorem, the mixed $F$-factorizaton of $j$ required for the weak lifting (see diagram (\ref{diagrama2})) can be chosen to be also an $F$-factorization in the Str\o{}m structure. Hence, the thesis holds.
\end{proof}

We are now ready to prove our main result.

\begin{Teorema}\label{principal}
Let $f\colon X\to Y$ be a map  in which $X$ and $Y$ are  of the homotopy type of locally compact CW-complexes. Then, $\csecat^{\topgen_M}(f)$ and $\csecat^{\topgen_Q}(f)$ are well defined and
$$
\csecat^{\topgen_S}(f)=\csecat^{\topgen_M}(f)=\csecat^{\topgen_Q}(f).
$$
\end{Teorema}

\begin{proof} First, note that $\csecat^{\topgen_S}$ is well defined as both, in the free and pointed setting, the Str{\o}m   structure satisfies the cube axiom \cite[\S5]{ma} and it is a $J$-category when restricted to the full subcategory of well pointed spaces.  Also, it is straightforward to check that the full subcategory of $\topgen_M$ consisting of the topological spaces with the same homotopy type of a CW-complex is a J-category, except that it is not pointed in the free setting. Therefore, by \cite[Thm. 15]{digargarmure}, $\csecat^{\topgen_M}(f)$ is also well defined for any $f$ as in the statement since $\gsecat^{\topgen_M}(f)=\wsecat^{\topgen_M}(f)$.

We now work with the Ganea version of sectional category to prove that $\csecat^{\topgen_S}(f)=\csecat^{\topgen_M}(f)$. By hypothesis, we may apply Lemma \ref{join-inv2} and Remark \ref{oss-cl-push} to prove inductively that for each $n\ge 0$, the $n$th join $\ast_{Y}^{n} X$ is cofibrant, i.e., it has the homotopy type of a CW-complex. Moreover, the join map $j_{n} \colon \ast_{Y}^{n} X \rightarrow Y$ can be chosen to be the same in both the Str{\o}m and the mixed structure. Then, by  Lemma \ref{weak-sec-uguale-str-mix1}, $j_{n}$ admits a weak section in the Str\o{}m structure if and only if it admits a weak section in the mixed one. Hence, the claim follows.

Next, we see that $\csecat^Q(f)$ is well defined and $\csecat^{\topgen_M}(f)=\csecat^{\topgen_Q}(f)$. For it, consider the adjunction given by the identity functors
\begin{displaymath}
\Id \colon \topgen_Q \leftrightarrows \topgen_M \colon \Id,
\end{displaymath}
 and observe that they satisfy the hypothesis of Proposition \ref{dim-mod-funt} ($\Id \colon \topgen_Q \to \topgen_M$ is the left adjoint). Hence, by Theorem \ref{teor-TC-uguale}, we obtain that for any morphism $g$ in $\topgen$, not necessarily under the hypothesis of the theorem,
$$
\gsecat^{\topgen_M}(g)=\gsecat^{\topgen_Q}(g)\quad\text{and}\quad \wsecat^{\topgen_M}(g)=\wsecat^{\topgen_Q}(g).
$$
Thererefore, with $f$ as in the theorem, $\gsecat^{\topgen_M}(f)=\wsecat^{\topgen_M}(f)$ and the claim follows.  In the pointed case note that there is no difference between $\csecat^{\topgen_M}$ and the one computed in the full subcategory of well pointed topological spaces of the homotopy type of CW-complexes.
\end{proof}
\begin{Osservazione}
We have explicitely shown that, with $f$ as in Theorem \ref{principal},  \ $\csecat^{\topgen_M}(f)$ and $\csecat^{\topgen_Q}(f)$ are well defined as at this point we are not aware of whether the Quillen or mixed structure on $\topgen$ satisfies the cube axiom or not.

On the other hand, if the reader wishes to work only with the Ganea version of sectional category, the locally compactness hypothesis in Theorem \ref{principal} is not necessary.
\end{Osservazione}

\begin{Osservazione}\label{veremos}
As asserted in the Introduction,
$$\secat(f)=\csecat^{\topgen_S}(f),
 $$
 for any map $f\colon X\to Y$ in $\topgen=\topo$ with $Y$ normal \cite[Thm. 2.2]{ferghienkahlvan}. Hence, in this case, and under the hypothesis of the theorem above, this invariant also coincides with $\csecat^{\topgen_Q}$.
Moreover, for path-connected spaces, and as long as $X$ is normal and well pointed in the first equality, and $X\times X$ is normal in the second, we have,
 $$
  \cat(X)=\ccat^{\topgen_S}(X)\quad\text{and}\quad  \tc(X)=\ctc^{\topgen_S}(X).
   $$
   Here, as stated in the Introduction, $\topgen=\topo_*$ for the first equality and $\topgen=\topo$ for the second. In this way, and to avoid excessive notation, whenever any of the classical notions $\secat(f)$, $\cat(X)$ or $\tc(X)$ appears henceforth, we implicitly make the above topological assumptions.
\end{Osservazione}
In view of this remark and  Theorem \ref{principal} we deduce:

 \begin{Corollario}\label{elunico} Let $X$ be a  space of the homotopy type of a locally compact CW-complex. Then,
 $$
 \cat(X)=\ccat^{\topgen_S}(X)=\ccat^{\topgen_Q}(X),\quad \tc(X)=\ctc^{\topgen_S}(X)=\ctc^{\topgen_Q}(X).
 $$
 \hfill$\square$
 \end{Corollario}

\begin{Esempio}\label{contro-ex}
In general, $\csecat^{\topgen_S}$ and $\csecat^{\topgen_Q}$  are different invariants. Indeed, let $X$ be a path-connected, non contractible space which has trivial homotopy groups (e.g. the Warsaw circle). Hence, both  $\Delta\colon X\to X\times X$ and $*\to X$ are  weakly equivalent to the trivial morphisms $\ast \rightarrow \ast \times \ast$, $*\to*$, and therefore $\ccat^{\topgen_Q}(X)=\ctc^{\topgen_Q}(X)=0$. On the other hand, as $X$ is non contractible,  $\ccat^{\topgen_S}(X),\ctc^{\topgen_S}(X)\ge  1$.
\end{Esempio}

As an application we finish by relating these invariants with the corresponding ones in the simplical category. Denote by $\simpset$ the category of either free or pointed simplicial sets endowed with the proper model structure in which the $fib=$Kan fibrations, $cof=$injective maps and $we=$homotopy equivalences of simplicial maps \cite{qui}. Consider the classical {\em realization} and {\em singular} pair of adjoint equivalences,
\begin{displaymath}
\mid \cdot \mid  \colon \simpset \leftrightarrows \topgen \colon \singular,
\end{displaymath}
where the  $\singular$ is the right adjoint.

\begin{Proposizione} For any morphism $f\in\topgen$,
$$
\gsecat^{\topgen_Q}(f)=\gsecat^{\simpset} \bigl(\singular(f)\bigr),\quad \wsecat^{\topgen_Q}(f)=\wsecat^{\simpset} \bigl(\singular(f)\bigr).
$$
\end{Proposizione}
\begin{proof} Simply observe that, when choosing the Quillen structure $\topgen_Q$, the adjoint functors $\mid \cdot \mid$ and $\singular$ trivially are modelization functors which satisfy the hypothesis of Theorem \ref{teor-TC-uguale}.
\end{proof}

Hence, applying directly Theorem \ref{principal} and Corollary \ref{elunico} we obtain.

\begin{Teorema}\label{ultimo}
Let $f$ be a map  between spaces of the homotopy type of locally compact CW-complexes. Then,
$\csecat^{\simpset} \bigl(\singular(f)\bigr)$ is well defined and
$$
\secat(f)=\csecat^{\simpset} \bigl(\singular(f)\bigr).
$$
Moreover, for any space $X$ of the homotopy type of a locally compact CW-complex
$$
 \cat(X)=\ccat^{\simpset}\bigl(\singular(X)\bigr)\quad\text{and}\quad \tc(X)=\ctc^{\simpset}\bigl(\singular(X)\bigr).
 $$
\hfill$\square$
\end{Teorema}

\bigskip

\noindent Dipartimento di Matematica,\\ Universit\`a di Pisa,\\
Largo B. Pontecorvo 5,
56127 Pisa, Italy.\\
\textit{E-mail address}: \textbf{moraschini@mail.dm.unipi.it}

\bigskip
\noindent Departamento de  Geometr{\'\i}a y Topolog{\'\i}a,\\
Universidad de M\'alaga,\\ Ap. 59,
29080 M\'alaga, Spain.\\
\textit{E-mail address}: \textbf{aniceto@uma.es}
\end{document}